\documentclass[a4paper, 12pt]{article}

\usepackage[utf8]{inputenc}
\usepackage{amsmath}
\usepackage{amsthm}
\usepackage{amsfonts}
\usepackage{amssymb}
\usepackage{amstext}
\usepackage{dsfont}
\usepackage{graphicx}
\usepackage{color}
\usepackage[colorlinks]{hyperref}
\usepackage{epigraph}
\usepackage{fullpage}
\usepackage{tikz}
\usetikzlibrary{arrows,shapes,matrix}

\graphicspath{{Figures/}{.}}%

\title{ A limiting random analytic  function related to the CUE}
\author{Reda \textsc{Chhaibi}        \footnote{\texttt{reda.chhaibi@math.uzh.ch}},
        Joseph \textsc{Najnudel}     \footnote{\texttt{joseph.najnudel@math.univ-toulouse.fr}},
        Ashkan \textsc{Nikeghbali}   \footnote{\texttt{ashkan.nikeghbali@math.uzh.ch}}
        } 

\allowdisplaybreaks[4]

\DeclareMathOperator{\Var}{Var}

\begin{document}

\def\half{\frac{1}{2}}

\def\l{\lambda}
\def\t{\theta}
\def\T{\Theta}
\def\m{\mu}
\def\a{\alpha}
\def\b{\beta}
\def\g{\gamma}
\def\o{\omega}
\def\p{\varphi}
\def\D{\Delta}
\def\O{\Omega}

\def\N{{\mathbb N}}
\def\Z{{\mathbb Z}}
\def\Q{{\mathbb Q}}
\def\R{{\mathbb R}}
\def\C{{\mathbb C}}

\def\P{{\mathbb P}}
\def\E{{\mathbb E}}

\def\Ac{{\mathcal A}}
\def\Bc{{\mathcal B}}
\def\Cc{{\mathcal C}}
\def\Dc{{\mathcal D}}
\def\Ec{{\mathcal E}}
\def\Fc{{\mathcal F}}
\def\Gc{{\mathcal G}}
\def\Hc{{\mathcal H}}
\def\Ic{{\mathcal I}}
\def\Kc{{\mathbb K}}
\def\Lc{{\mathcal L}}
\def\Oc{{\mathcal O}}
\def\Pc{{\mathcal P}}
\def\Qc{{\mathcal Q}}
\def\Rc{{\mathcal R}}
\def\Sc{{\mathcal S}}
\def\Tc{{\mathcal T}}
\def\Uc{{\mathcal U}}
\def\Zc{{\mathcal Z}}

\def\afrak{{\mathfrak a}}
\def\bfrak{{\mathfrak b}}
\def\gfrak{{\mathfrak g}}
\def\hfrak{{\mathfrak h}}
\def\kfrak{{\mathfrak k}}
\def\nfrak{{\mathfrak n}}
\def\pfrak{{\mathfrak p}}
\def\slfrak{{\mathfrak sl}}

\maketitle

\setlength{\footskip}{2cm}

\newtheorem{corollary}{Corollary}[section]
\newtheorem{question}{Question}[section]
\newtheorem{conjecture}{Conjecture}[section]
\newtheorem{definition}{Definition}[section]
\newtheorem{example}{Example}[section]
\newtheorem{lemma}{Lemma}[section]
\newtheorem{proposition}{Proposition}[section]
\newtheorem{properties}{Properties}[section]
\newtheorem{property}{Property}[section]
\newtheorem{rmk}{Remark}[section]
\newtheorem{thm}{Theorem}[section]

\newcommand{\Card}[1]{\textup{Card($#1$)} \xspace}

\begin{abstract}
We show in this paper that, when properly rescaled in time and in space, the characteristic polynomial of a random unitary matrix converges almost surely to a random analytic function whose  zeros, which are on the real line, form a determinantal point process with sine kernel. We prove this result in the framework of virtual isometries to circumvent the fact that the rescaled characteristic polynomial does  not even have a moment of order one, hence making the classical techniques of random matrix theory difficult to apply.
\end{abstract}

\section{Introduction}

A major breakthrough in the so called random matrix approach in number theory is  the seminal paper of Keating and Snaith \cite{bib:KS}, where they conjecture that the characteristic polynomial of a random unitary matrix, restricted to the unit circle, is a good and accurate model to predict the value distribution of the Riemann zeta function on the critical line.  In particular, using this philosophy, they were able to conjecture the exact asymptotics of the moments of the Riemann zeta function, a result which was considered to be out of reach with classical tools form analytic number theory. One simple and naive explanation for the success of the characteristic polynomial as a random model to the Riemann zeta function comes from Montgomery's conjecture that the zeros of the Riemann zeta function on the critical line (after rescaling) statistically behave like the eigenangles (after rescaling) of large random unitary matrices.  Moreover the limiting point process obtained from the eigenvalues is a determinantal point process with the sine kernel. A natural question which then naturally arose in the community was the existence of a random analytic function with zeros which from a determinantal point process with the sine kernel and  which would be obtained as a limiting object from characteristic polynomials. As we shall see below, the sequence of characteristic polynomials of random unitary matrices of growing dimensions does not converge. We shall nonetheless prove that after a proper rescaling in "time" (the characteristic polynomial can be viewed as a stochastic process with parameter $z\in\C$, and we shall consider the characteristic polynomial at the scale $z/n$) and space (that is we normalize with the value of the characteristic polynomial at $1$), this sequence converges locally uniformly on compact subsets of the complex plane to a random analytic function with the desired property. The convergence will be proved to occur almost surely, thanks to the use of virtual isometries introduced in \cite{bib:BNN}. The basic idea behind virtual isometrics is that of coupling the different dimensions of the unitary groups $U(n)$ together in such a way that marginal distribution on each $U(n)$, for fixed $n$, is the Haar measure. This strong convergence will in turn imply the weak convergence of the same objects. But since our rescaled characteristic polynomials do not even have a moment of order one, proving the weak convergence as stated in Theorem \ref{thm:maininlaw} with classical methods does not seem to be an easy task. On the other hand combining some of the fine estimates on the eigenvalues from \cite{bib:MNN} which make strong use of  the coupling from virtual isometries and some other classical estimates on sine kernel determinantal point processes is enough to establish almost sure convergence. In the sequel, we introduce the main objects and notation and state our main theorems.

In the above mentioned paper by Keating and Snaith \cite{bib:KS}, the authors computed the moments of the characteristic polynomial of a random unitary 
matrix following the Haar measure. They deduced that the characteristic polynomial asymptotically behaves like a 
log-normal distributed random variable when the dimension $n$ goes to infinity: more precisely, 
its logarithm, divided by $\sqrt{\log n}$, tends to a complex Gaussian random variable 
$Z$ such that $\mathbb{E}[Z] = \mathbb{E} [Z^2] = 0$ and $\mathbb{E} [|Z|^2]= 1$. 
This result has been generalized in Hughes, Keating and O'Connell \cite{bib:HKO}, where the authors proved the asymptotic
independence of the characteristic polynomial taken at different fixed points. A question which then naturally arises  concerns the behavior of the characteristic polynomial at points which vary with the 
dimension and which are sufficiently close to each other in order to avoid asymptotic independence. 
The scale we consider in the present paper is the average spacing of the eigenangles of a unitary
matrix in 
dimension $n$, i.e. $2 \pi/n$. 
More precisely, let $(U_n)_{n \geq 1}$ be a sequence of matrices, $U_n$ being Haar-distributed in $U(n)$, and let 
$Z_n$ be the characteristic polynomial of $U_n$: 
$$Z_n(X) = \operatorname{det} ( X - U_n).$$
For a given $z \in \mathbb{C}$, we consider the value of $Z_n$ at the two points $1$ and $e^{2 i z \pi/n}$, 
whose distance is equivalent to $2 \pi |z|/n$ when $n$ goes to infinity. We know that the law of $Z_n(1)$ can be approximated 
by the exponential of a gaussian variable of variance $\log n$, so it does not converge when $n$ goes to infinity: the same is true for $Z_n(e^{2 i z \pi/n})$. 
In order to obtain a convergence in law, it is then natural to consider the ratio $Z_n(e^{2i z \pi/n})/ Z_n(1)$,
which has order of magnitude $1$ and which is well-defined as soon 
as $1$ is not an eigenvalue of $U_n$, an event occuring almost surely. 

If we consider all the values of $z$ together, we obtain a random entire function $\xi_n$, defined by 
$$\xi_n (z) = \frac{Z_n(e^{2i z \pi/n})}{Z_n(1)}.$$
We will prove that this function has a limiting distribution when $n$ goes to infinity. More precisely, the 
main result of this article is the following: 
\begin{thm} \label{thm:maininlaw}
 In the space of continuous functions from $\mathbb{C}$ to $\mathbb{C}$, endowed with the topology of uniform 
 convergence on compact sets, the random entire function $\xi_n$ converges in law to a limiting 
 entire function $\xi_{\infty}$. The zeros of $\xi_{\infty}$ are all real and form a determinantal sine-kernel point 
 process,
 i.e. for all $r \geq 1$, the $r$-point correlation function $\rho_r$ corresponding to this point process is 
 given, for all $x_1, \dots, x_r \in \mathbb{R}$, by 
 $$\rho_r (x_1, \dots, x_r) = \operatorname{det} \left( \frac{\sin [\pi(x_j - x_k)]}{ \pi(x_j - x_k)} \right)_{1 \leq 
 j,k  \leq r}. $$
 \end{thm} 
 Taking a finite number of points $z_1, \dots, z_p \in \mathbb{C}$, we see in particular that the joint law of 
 the mutual ratios of $Z_n (e^{2 i \pi z_1/n}), \dots, Z_n (e^{2 i \pi z_p/n})$ converges when $n$ goes to infinity. 
 
 In order to prove Theorem \ref{thm:maininlaw}, we will define the sequence $(U_n)_{n \geq 1}$ of 
 unitary matrices in a common probability space, with a coupling chosen in such a way that an almost sure convergence 
 occurs. An interest of this method is that it is more convenient to deal with pointwise convergence than with 
 convergence in law when we work on a functional space. Moreover, the coupling gives 
 a powerful way to track the sequence $(\xi_n)_{n \geq 1}$ of holomorphic function, and a 
 deterministic link between this sequence and the limiting function $\xi_{\infty}$. 
Besides it is important to stress that the moments method, which is a classical technique in random matrix theory, is  
impossible to implement. Indeed the random function at hand $\xi_n$ does not have any integer moment when evaluated on circle, which 
makes the use of the formulas on moments of ratios in \cite{bib:BG} and \cite{bib:CFZ08} difficult to use. For 
example, in Theorem 3 of the article \cite{bib:BG}, one clearly sees the divergence of ratios, as the evaluation 
points get close to $1$.
 The coupling we consider here corresponds to the notion of {\it virtual isometries}, as defined by Bourgade, Najnudel 
 and Nikeghbali in \cite{bib:BNN}. The sequence $(U_n)_{n \geq 1}$ can be constructed in the following way: 
 \begin{enumerate}
 \item One considers a sequence $(x_n)_{x \geq 1}$ of independent random vectors, $x_n$ being uniform on the unit
 sphere of 
 $\mathbb{C}^n$. 
 \item Almost surely, for all $n \geq 1$, $x_n$ is different from the last basis vector $e_n$ of $\mathbb{C}^n$, which 
 implies that there exists a unique $R_n \in U(n)$ such that $R_n(e_n) = x_n$ and $R_n - I_n$ has rank one. 
 \item We define $(U_n)_{n \geq 1}$ by induction as follows: $U_1 = x_1$ and for all $n \geq 2$, 
 $$U_n = R_n \left( \begin{array}{cc}
U_{n-1} & 0\\
0 & 1 \end{array} \right).$$
 \end{enumerate}
It has already been proven in \cite{bib:BHNY} that with this construction, $U_n$ follows, for all $n \geq 1$, the 
Haar measure on $U(n)$. From now on, we always assume that the sequence $(U_n)_{n \geq 1}$ is defined with this coupling.

For each value of $n$, let $\lambda_1^{(n)}, \dots, \lambda_n^{(n)}$ be the eigenvalues of $U_n$, ordered 
counterclockwise, starting from $1$: they are almost surely pairwise distinct and different from $1$. 
If $1 \leq k \leq n$, we denote by $\theta_k^{(n)}$ the argument of $\lambda_k^{(n)}$, taken 
in the interval $(0, 2\pi)$: $\theta_k^{(n)}$ is the $k$-th strictly positive eigenangle of $U_n$. 
If we consider all the eigenangles of $U_n$, taken not only in $(0, 2 \pi)$ but in the whole real line, 
we get a $(2 \pi)$-periodic set with $n$ points in each period.  If the eigenangles are 
indexed increasingly by $\mathbb{Z}$, 
we obtain a sequence
$$ \dots < \theta_{-1}^{(n)} < \theta_{0}^{(n)} < 0 < \theta_{1}^{(n)} < \theta_{2}^{(n)} < \dots, $$
for which $\theta_{k+n}^{(n)} = \theta_k^{(n)} + 2 \pi$ for all $k \in \mathbb{Z}$. 

It is also convenient to extend the sequence of eigenvalues as a $n$-periodic sequence indexed 
by $\mathbb{Z}$, in such a way that for all $k \in \Z$, 
$$ \lambda_k^{(n)} = \exp\left( i\theta_k^{(n)} \right). $$

With the notation above, the following holds:
\begin{thm}[Theorem 7.3 in \cite{bib:MNN}]
\label{thm:as_cv_sine}
Almost surely, the point process 
$$\left( y_k^{(n)} := \frac{n}{2 \pi} \theta_{k}^{(n)} , \ k \in \Z \right)$$
converges pointwise to a determinantal sine-kernel point process $\left( y_k , \ k \in \Z \right)$. And 
moreover, almost surely, the following estimate holds for all $\varepsilon  > 0$:
$$ \forall  k \in [-n^{\frac{1}{4}},n^{\frac{1}{4}}],  \ y_k^{(n)} = y_k + O_\varepsilon\left( (1+k^2) n^{-\frac{1}{3}+\varepsilon} \right)  $$
\end{thm}
\begin{rmk}
 The implied constant in $O_\varepsilon$ is random: more precisely, it may depend on the sequence $(U_m)_{m \geq 1}$ and 
 on $\varepsilon$. However, it does not depend on $k$ and $n$. 
\end{rmk}
 The remaining of the paper is devoted to the proof of the following strong convergence result and to some properties of the limiting function obtained there.
 \begin{thm} \label{thm:main}
Almost surely and uniformly on compact subsets of $\C$, we have the convergence:
$$ \xi_n\left( z \right) \stackrel{ n \rightarrow \infty}{\longrightarrow} \xi_\infty(z) := e^{i \pi z} 
\prod_{k \in \Z}\left( 1 - \frac{z}{y_k}\right)$$
Here, the infinite product is not absolutely convergent. It has to be understood as the limit of the following 
product, obtained by regrouping the factors  two by two: 
$$\left( 1 - \frac{z}{y_0} \right) \prod_{k \geq 1} \left[ \left( 1 - \frac{z}{y_k} \right)
\left( 1 - \frac{z}{y_{-k}} \right) \right],$$
which is absolutely convergent. 
\end{thm}
This theorem immediately implies Theorem \ref{thm:maininlaw}, provided that $\xi_{\infty}$ is entire and 
that the zeros of $\xi_{\infty}$ are exactly given by the sequence $(y_k)_{k \in \Z}$. 
 The first point is a direct consequence of the fact that 
 $\xi_{\infty}$ is the uniform limit on compact sets of the sequence of entire functions $(\xi_n)_{n \geq 1}$, and 
 the second point is a consequence of the fact that the $k$-th factor of the absolutely convergent product above
 vanishes at $y_k$ and $y_{-k}$ and only at these points.

The proof of Theorem \ref{thm:main}  will 
be provided in Section \ref{section:convergence}, using estimates on the argument of $Z_n$, stated in Section 
\ref{section:estimatesZn}, and estimates on the renormalized eigenangles $y_k^{(n)}$, stated 
in Section \ref{section:estimatesykn}. 
In Section \ref{section:properties}, we prove some properties of the limiting random function $\xi_{\infty}$, and 
in Section \ref{section:open}, we conclude with some remarks and potential problems of interest on the random analytic function $\xi_\infty(z)$.

\section{Some estimates on the argument of \texorpdfstring{$Z_n$}{Zn} } 
\label{section:estimatesZn}
In this section, we study the argument of $Z_n$, in order to deduce estimates on the deviation of $y_k^{(n)}$ from $k$. 

Here, we define the argument as the imaginary part of $\log Z_n$, where the determination of 
the logarithm is the only one such that $\log Z_n$ is continuous on the following maximal simply connected domain:
$$ \Dc := \C \backslash \left\{ r e^{i \theta_k^{(n)}} | k \in \Z, r \geq 1 \right\},$$
and such that $\log Z_n(0) \in i (-\pi, \pi]$  (note that $|Z_n(0)| = |(-1)^n \operatorname{det} (U_n)| = 1$, so 
$\log Z_n(0)$ should be purely imaginary).  

For all $z \in \Dc$, we have 
$$ \log Z_n(z) = \log Z_n(0) + \sum_{k=1}^n \log\left( 1 - \frac{z}{\lambda_k^{(n)}} \right),$$
where the principal branch of the logarithm is considered. 

The next proposition gives a link between the number of eigenvalues of $U_n$ in a given arc of circle, and 
the variation of the argument of $Z_n$ along this arc. 
 The derivation is relatively standard and follows almost verbatim the usual computation for the Riemann
 zeta function (for example, see \cite{bib:Titchmarsh}, p. 212). Another proof of the same result is 
 given in Hughes \cite{bib:HughesPhD}, p. 35-36. 
\begin{proposition}
Consider $A$ and $B$ two points on the unit circle. Note $\stackrel{\frown}{AB}$ for the arc joining $A$ and $B$ counterclockwise. Denote by $\ell\left( \stackrel{\frown}{AB} \right)$ the length of the arc and $N\left( \stackrel{\frown}{AB} \right)$ the number of zeros of $Z_n$ in the arc. We assume that $A$ and $B$ are not zeros of $Z_n$. Then:
$$ N\left( \stackrel{\frown}{AB} \right) = \frac{ n \ell\left( \stackrel{\frown}{AB} \right) }{2 \pi}
                 - \frac{1}{\pi}\left[  \Im \log\left(Z_n(B)\right) - \Im \log\left(Z_n(A)\right) \right].$$
\end{proposition}
\begin{rmk}
This shows that the imaginary part of the determination of the logarithm $\Im \log Z_n(z)$ increases with speed $n/2$ and jumps by $-\pi$ when encountering a zero.
\end{rmk}
\begin{proof}
Consider a contour $C_{A,B}$ defined as follows. 

Let $0 < r < 1$. $C_{A,B}$ is the contour joining the points:
$$ A r, \frac{A}{r}, \frac{B}{r}, B r$$
The pairs $\left( Ar, \frac{A}{r} \right)$ and $\left( Br, \frac{B}{r} \right)$ are joined with straight lines, while the other two are joined by circular arcs. 

A standard residue calculus shows that:
$$ N\left( \stackrel{\frown}{AB} \right) = \frac{1}{2 i \pi} \int_{C_{A,B}} \frac{Z_n'(z)}{Z_n(z)} dz $$

One can split the contour $C_{A,B}$ into two contours $(D_{A,B}, D_{A,B}')$ symmetric with respect to the involution $z \mapsto \frac{1}{\overline{z}}$. $D_{A,B}$ is chosen to be the inner contour. Hence:
\begin{align*}
  & \int_{C_{A,B}} \frac{Z_n'(z)}{Z_n(z)} dz\\
= & \int_{D_{A,B}} \frac{Z_n'(z)}{Z_n(z)} dz + \int_{D_{A,B}'} \frac{Z_n'(z)}{Z_n(z)} dz\\
= & \int_{D_{A,B}} \frac{Z_n'(z)}{Z_n(z)} dz - \int_{D_{A,B}} \frac{Z_n'(\frac{1}{\overline{z}})}{Z_n(\frac{1}{\overline{z}})} d\left( \frac{1}{\overline{z}} \right)\\
= & \int_{D_{A,B}} \frac{Z_n'(z)}{Z_n(z)} dz - \overline{ \int_{D_{A,B}} \frac{\overline{Z_n}'(\frac{1}{z})}{\overline{Z_n}(\frac{1}{z})} d\left( \frac{1}{z} \right) }
\end{align*}
The minus sign appearing in the previous equation is explained by the fact that $z \mapsto \frac{1}{\overline{z}}$ is not orientation preserving,
and $\overline{Z_n}$ denotes the polynomial obtained by taking the conjugate of each 
coefficient of $Z_n$.
Now, the polynomial $Z_n$ satisfies the following functional equation: 
\begin{align}
\label{eqn:functional_eq}
\overline{Z_n}\left( \frac{1}{X} \right) = X^{-n} (-1)^n \det\left( U_n^{-1} \right) Z_n\left( X \right)
\end{align}
Taking the logarithmic derivative, one obtains:
$$ \frac{\overline{Z_n}'(1/z)}{ \overline{Z_n}(1/z)} d \left( \frac{1}{z} \right) = -n \frac{dz}{z} + \frac{ Z_n'(z)}{Z_n(z)} dz$$

Therefore:
$$ \int_{C_{A,B}} \frac{Z_n'(z)}{Z_n(z)} dz = 2i \Im\left( \int_{D_{A,B}} \frac{Z_n'(z)}{Z_n(z)} dz \right) + n \overline{\int_{D_{A,B}} \frac{dz}{z} }$$
Moreover, $i \int_{D_{A,B}} \frac{dz}{z}$ is the length of the arc that $C_{A,B}$ encircles, which is exactly $\ell\left( \stackrel{\frown}{AB} \right)$. Equivalently:
$$ \int_{D_{A,B}} \frac{dz}{z} = -i \ell\left( \stackrel{\frown}{AB} \right)$$
Hence:
$$ \int_{C_{A,B}} \frac{Z_n'(z)}{Z_n(z)} dz = 2i \Im\left( \int_{D_{A,B}} \frac{Z_n'(z)}{Z_n(z)} dz \right) + i n \ell\left( \stackrel{\frown}{AB} \right)$$
After dividing by $2 \pi i$:
$$ N\left( \stackrel{\frown}{AB} \right) = \frac{ n \ell\left( \stackrel{\frown}{AB} \right) }{2 \pi} + \frac{1}{\pi}\Im\left( \int_{D_{A,B}} \frac{Z_n'(z)}{Z_n(z)} dz \right).$$
Consequently:
$$ N\left( \stackrel{\frown}{AB} \right) = \frac{ n \ell\left( \stackrel{\frown}{AB} \right) }{2 \pi} - \frac{1}{\pi}\Im\left( \log\left( Z_n(B)\right) 
                                        - \log\left( Z_n( A )\right) \right).$$
\end{proof}

\begin{corollary} \label{corollary:kykn}
Let $k \in \Z$, and let $\varepsilon>0$ be small enough so that there are no eigenangle of $U_n$ in
$[0, \varepsilon]$ and $(\theta_k^{(n)}, \theta_k^{(n)} + \varepsilon]$. Then:
$$ k = y_k^{(n)} - \frac{1}{\pi}\Im\left( \log\left( Z_n( e^{i (\theta_k^{(n)} + \varepsilon) } )\right) 
                                        - \log\left( Z_n( e^{i \varepsilon } )\right) \right)$$
\end{corollary}
\begin{proof}
Notice first that if $k$ is increased by $n$, $\theta_k^{(n)}$ increases by $2 \pi$, $y_k^{(n)}$ increases by 
$n$, $\lambda_k^{(n)} = e^{i \theta_k^{(n)}}$ does not change, and the assmption made on $\varepsilon$ remains the same. 
Hence, in the equality we want to prove, the right-hand side and the left-hand side both increase by $n$, which implies 
that it is sufficient to show the corollary for $1 \leq k \leq n$. 
If these inequalities are satisfied, let us choose, in the previous proposition, 
$A = e^{i \varepsilon}$ and $B = e^{i (\theta_k^{(n)} + \varepsilon) }$. Then the contour defined in
the proof of the proposition encircles exactly the $k$ first eigenvalues:
$$ \left( \lambda_1^{(n)}, \lambda_2^{(n)}, \dots, \lambda_k^{(n)} \right).$$
Hence, we have:
$$ N\left( \stackrel{\frown}{AB} \right) = k,$$
and 
$$ \frac{ n \ell\left( \stackrel{\frown}{AB} \right) }{2 \pi} = \frac{ n \theta_k^{(n)} }{2 \pi} = y_k^{(n)},$$
which proves the corollary. 
\end{proof}

This corollary shows that it is equivalent to control the argument of $Z_n$, and the distance between $k$ and 
$y_k^{(n)}$. In the remaining of this section, we give some explicit bounds on the distribution of 
$\Im \log (Z_n)$ on the unit circle. 
\begin{proposition}
For all $x > 0$, one has 
$$\P\left( |\Im\left( \log Z_n(1) - \log Z_n(0) \right)| \geq x \right) \leq 2\exp\left( - \frac{x^2}{C + \log n}\right),$$
 where $C > 0$ is a universal constant. 
\end{proposition}
\begin{rmk}
 In the proof below, we prove that one can take $C = \frac{\pi^2}{6} + 1$. 
\end{rmk}

\begin{proof}
Let us note
$$ X_n = \Im\left( \log Z_n(1) - \log Z_n(0) \right)$$
Thanks to the formula (1.1) in \cite{bib:BHNY}:
$$ \forall \lambda \in \R, \E\left( e^{ \lambda X_n} \right) = \prod_{k=1}^n \frac{ \Gamma\left( k \right)^2}{ \Gamma\left( k+\frac{i\lambda}{2} \right) \Gamma\left( k-\frac{i\lambda}{2} \right) }$$
Let us start with the standard Chernoff bound:
$$ \forall \lambda > 0, \P\left( X_n \geq x \right) \leq e^{- \lambda x } \E\left( e^{\lambda X_n}\right).$$
Now, using the infinite product formula for the Gamma function:
$$ \forall z \in \C, \frac{1}{\Gamma(z)} = e^{\gamma z} z \prod_{j=1}^\infty \left( 1 + \frac{z}{j} \right) e^{-z/j},$$
\begin{align*}
\E\left( e^{\lambda X_n}\right) & = \prod_{k=1}^n \frac{ \Gamma\left( k \right)^2}{ \Gamma\left( k+\frac{i\lambda}{2} \right) \Gamma\left( k-\frac{i\lambda}{2} \right) }\\
& = \prod_{k=1}^n \left( \frac{k^2 + \frac{\lambda^2}{4} }{k^2} \prod_{j=1}^\infty \frac{ \left( 1 + \frac{k+\frac{i\lambda}{2}}{j} \right)\left( 1 + \frac{k-\frac{i\lambda}{2}}{j} \right) }{ \left( 1 + \frac{k}{j} \right)^2 } \right)\\
& = \prod_{k=1}^n \left( \frac{k^2 + \frac{\lambda^2}{4} }{k^2} \prod_{j=1}^\infty \frac{ \left( j + k+\frac{i\lambda}{2} \right)\left( j + k-\frac{i\lambda}{2} \right) }{ \left( j + k \right)^2 } \right)\\
& = \prod_{k=1}^n \prod_{j=0}^\infty \frac{ \left(j + k\right)^2 +\frac{\lambda^2}{4} }{ \left( j + k \right)^2 }\\
& = \prod_{k=1}^n \prod_{j=0}^\infty \left( 1 + \frac{\lambda^2}{4\left(j + k\right)^2} \right)\\
& \leq \exp\left( \sum_{k=1}^n \sum_{j=0}^\infty \frac{\lambda^2}{4\left(j + k\right)^2} \right)\\
& = \exp\left( \frac{\lambda^2}{4} \sum_{k=1}^n \sum_{j=k}^\infty \frac{1}{j^2} \right)\\
& \leq \exp\left( \frac{\lambda^2}{4} \sum_{k=1}^n \left( \frac{1}{k^2} + \int_k^\infty \frac{dt}{t^2} \right) \right)\\
& = \exp\left( \frac{\lambda^2}{4} \sum_{k=1}^n \left( \frac{1}{k^2} + \frac{1}{k} \right) \right)\\
& \leq \exp\left( \frac{\lambda^2}{4} \left( \frac{\pi^2}{6} + 1 + \log n \right) \right)
\end{align*}
Eventually  for $C = \frac{\pi^2}{6} + 1$, we obtain
$$ \P\left( X_n \geq x \right) \leq \min_{\lambda > 0} e^{- \lambda x + \frac{\lambda^2}{4} \left( C + \log n \right) }. $$
The minimum is reached for $\lambda = \frac{2x}{C + \log n}$, giving us the bound:
$$ \P\left( \Im\left( \log Z_n(1) - \log Z_n(0) \right) \geq x \right) \leq \exp\left( - \frac{x^2}{C + \log n}\right).$$
The desired bound is obtained from the symmetry of $\Im\left( \log Z_n(1) - \log Z_n(0) \right)$, as eigenvalues are invariant in law under conjugation:
\begin{align*}
  & \P\left( |\Im\left( \log Z_n(1) - \log Z_n(0) \right)| \geq x \right)\\
= & \P\left( \Im\left( \log Z_n(1) - \log Z_n(0) \right) \geq x \right) + \P\left( -\Im\left( \log Z_n(1) - \log Z_n(0) \right) \geq x \right)\\
= & 2 \P\left( \Im\left( \log Z_n(1) - \log Z_n(0) \right) \geq x \right)
\end{align*}
\end{proof}
We deduce the following estimate on the maximum of the argument of $Z_n$ on the unit circle:  

\begin{proposition} \label{proposition:maxZn}
Almost surely:
$$ \sup_{ |z| = 1 , z \in \Dc } \left| \Im \log Z_n(z) \right| = O\left( \log n \right)$$
More precisely, for any $D > \sqrt{2}$:
$$ \exists n_0 \in \N, \forall n \geq n_0, \sup_{ |z| = 1 , z \in \Dc } \left| \Im \log Z_n(z) \right| \leq D \log n$$
which means that almost surely:
$$ \limsup_{n} \frac{1}{\log n} \sup_{ |z| = 1 , z \in \Dc } \left| \Im \log Z_n(z) \right| \leq \sqrt{2} $$
\end{proposition}
\begin{proof}
Consider $n$ regularly spaced points on the circle, say:
$$ x_{k,n} := e^{ i \frac{2\pi k}{n} }, \quad k = 0, 1, 2, \dots , n-1,$$
and the events:
$$ A_{k,n} := \left\{ |\Im \log Z_n\left( x_{k,n}\right)-\Im \log Z_n\left(0\right)| \geq D \log n \right\}$$
Because the law of the spectrum of $U_n$ is invariant under rotation, all the events $A_{k,n}$ have the same probability for different $k$'s. Moreover, thanks to the previous Chernoff bound:
\begin{align*}
n \P\left( A_{0,n} \right) & \leq 2 n \exp\left( - \frac{D^2 (\log n)^2 }{C + \log n}\right) \\
                           & \leq 2 n \exp\left( - D^2 \left( \log n - C \right)\right) \\
                           & \leq 2 e^{D^2 C} n^{1-D^2}
\end{align*}
Hence:
$$ \sum_{n=1}^\infty \sum_{k=1}^n \P\left( A_{k,n} \right) = \sum_{n=1}^\infty n \P\left( A_{0,n} \right) < \infty $$
The Borel-Cantelli lemma ensures that, almost surely:
$$ \exists n_0 \in \N, \forall n \geq n_0, \forall k, \quad |\Im \log Z_n\left( x_{k,n}\right)| \leq \pi + D \log n$$

Now consider a point $z = e^{i\theta} \in \Dc$. For fixed $n$, it lies on the arc between $x_{k,n}$ and $x_{k+1,n}$ for a certain $k$. Because
$$ \theta \mapsto \Im \log Z_n(e^{i\theta})$$
is piece-wise linear, increasing with speed $n/2$ and only jumping by $-\pi$, we have:
$$ \Im \log Z_n(e^{i\theta}) \leq \Im \log Z_n(x_{k,n}) + \frac{n}{2}\left( \theta - \frac{2\pi k}{n}\right) \leq \Im \log Z_n(x_{k,n}) + \pi$$
In the other direction, we have
$$ \Im \log Z_n(e^{i\theta}) \geq \Im \log Z_n(x_{k+1,n}) - \frac{n}{2}\left( \frac{2\pi (k+1)}{n} - \theta \right) \geq \Im \log Z_n(x_{k+1,n}) - \pi$$
So that, almost surely:
$$ \exists n_0 \in \N, \forall n \geq n_0, \forall z \in \Dc, \quad |\Im \log Z_n\left( z \right)| \leq 2\pi + D \log n$$
The more precise estimate  $|\Im \log Z_n\left( z \right)| \leq D \log n$ follows after replacing $D$ by $D' \in (\sqrt{2}, D)$ in the previous computation
and considering $n_0$ large enough so that $2\pi < (D-D') \log n$.
\end{proof}

\section{Precise estimates for the eigenvalues of virtual isometries} 
\label{section:estimatesykn}

The following estimate will reveal crucial  for the proof of Theorem \ref{thm:main}. 

\begin{proposition}
\label{proposition:key_estimate}
Almost surely and uniformly in $n$ and $k$:
$$ y_k^{(n)} = k + O\left( \log(2+|k|) \right)$$ 
\end{proposition}

In fact, when $n=\infty$, this estimate is already easily deduced from existing literature (for example \cite{bib:meckes}, \cite{bib:soshnikov02}):
\begin{lemma}
\label{lemma:lemma_mm_sosh}
Almost surely:
$$ \forall k \in \Z, y_k = k + O\left( \log(2+|k|) \right)$$
\end{lemma}
\begin{proof}
Consider a sine-kernel process $y_k$. Let $X_{[a,b]}$ be the number of particles $y_k$ in $[a,b]$. 

Fix $A >0$. Let $X_A$ be the number of particles $y_k$ in $[0,A]$. Thanks to Proposition 2 in \cite{bib:meckes} (which is by the way also a standard result in the theory of point processes), which
can be applied to the sine-kernel process, $X_A$ is a sum of independent Bernoulli random variables. As in corollary 4 in \cite{bib:meckes}, we can deduce, using the
Bernstein inequality that
$$ \forall t>0, \P\left( |X_A - A| \geq t\right) \leq 2 \exp\left( -\min( \frac{t^2}{4 \Var(X_A)}, \frac{t}{2})\right).$$

An estimate for the variance is proved by Costin and Lebowitz \cite{bib:CL}  (see also Soshnikov \cite{bib:soshnikov02}):
$$ \Var(X_A) = \frac{1}{\pi^2} \log A + O(1)$$

Therefore, for all $D > 0$, 
$$ \P\left( |X_A - A| \geq D \log A \right) \leq 2 \exp\left( -(\log A) \min( \frac{D^2 \pi^2}{4 + O(1/\log A)}, \frac{D}{2})\right).$$

For $D > 2$, and $A$ large enough, $D^2 \pi^2/[4 + O(1/\log A)] > D/2$, which implies: 
$$\P\left( |X_A - A| \geq D \log A \right) \leq 2\exp\left( -(\log A)(D/2)\right) = 2 A^{-D/2}.$$
This quantity is summable for positive integer values of $A$. By Borel-Cantelli's lemma, we deduce that almost surely, 
for $A \in \mathbb{N}$:
$$ X_A = A + O\left( \log (2 +  |A|)  \right).$$
From the inequality
$$X_{[0, \lfloor A \rfloor ]} \leq X_{[0,A]} \leq X_{[0, \lceil A \rceil ]},$$
we deduce that the estimate remains true for all $A \geq 0$. 
Taking $A = y_k$ for $k > 0$ proves the proposition for positive indices. With the same argument one 
handles the negative ones.
\end{proof}

In order to prove proposition \ref{proposition:key_estimate}, we will also need the two lemmas:
\begin{lemma}
\label{lemma:lemma_1}
Almost surely:
$$ \forall k \in \Z, y_k^{(n)} = k + O\left( \log n \right)$$
\end{lemma}
\begin{proof}
This is an immediate consequence of Corollary \ref{corollary:kykn} and Proposition \ref{proposition:maxZn}. 
\end{proof}

\begin{lemma}
\label{lemma:lemma_2}
For every $0 < \eta < \frac{1}{6}$, there exists $\varepsilon > 0$ such that, almost surely:
$$ \forall k \in [- n^\eta,  n^\eta], y_k^{(n)} = y_k + O\left( n^{-\varepsilon}\right)$$
\end{lemma}
\begin{proof}
Since $k \in [- n^{1/4},  n^{1/4}] $, we can apply Theorem \ref{thm:as_cv_sine}, which gives, for all $\delta > 0$, 
$$ y_k^{(n)} = y_k + O_{\delta} \left((1+k^2) n^{-\frac{1}{3} + \delta} \right).$$
Since $k = O(n^{\eta})$, 
$$ y_k^{(n)} = y_k +  O_{\delta} \left(n^{2 \eta - \frac{1}{3} + \delta} \right),$$
which, by taking 
$$\delta = \frac{1}{6} - \eta > 0,$$
gives the desired result, for 
$$\varepsilon = - 2 \eta + \frac{1}{3} - \delta = 2 \delta - \delta = \delta > 0.$$
\end{proof}

\begin{proof}[Proof of Proposition \ref{proposition:key_estimate}]
In the range $|k| \geq n^{1/7}$, it is a consequence of Lemma \ref{lemma:lemma_1}. In the 
range $|k| < n^{1/7}$, it is a consequence of Lemmas \ref{lemma:lemma_mm_sosh} and \ref{lemma:lemma_2}
(for $\eta = 1/7$). 
\end{proof}

\section{Proof of Theorem \ref{thm:main}}
\label{section:convergence}

First, let us express $\xi_n$ in function of the renormalized eigenangles of $U_n$.
\begin{proposition}
One has 
$$ \xi_n\left( z \right) =  e^{i \pi z} \prod_{k \in \Z}\left( 1 - \frac{z}{y_k^{(n)}}\right),$$
where the infinite product has to be understood as the limit of the product from $k = -A$ to $k = A$ when the 
integer $A$ goes to infinity. 
\end{proposition}
\begin{proof}
\begin{align*}
\xi_n\left( z \right) & = \frac{Z_n\left( \exp( \frac{i 2 \pi z }{n}) \right)}{Z_n(1)}\\
                      & = \prod_{k=1}^n \frac{\exp( \frac{i 2 \pi z }{n}) - \lambda_k^{(n)}}{1 - \lambda_k^{(n)}}\\
                      & = \prod_{k=1}^n \frac{\exp( \frac{i 2 \pi z }{n}) - \exp\left( i\theta_k^{(n)} \right) }{1 - \exp\left( i\theta_k^{(n)} \right) }\\
                      & = \prod_{k=1}^n \frac{ \exp( \frac{i 2 \pi z }{2n} + \half i\theta_k^{(n)}) } { \exp( \half i\theta_k^{(n)})}
                                        \frac{ \exp( \frac{i 2 \pi z }{2n} - \half i\theta_k^{(n)}) - \exp\left( \half i\theta_k^{(n)} - \frac{i 2 \pi z }{2n} \right) }
                                             { \exp\left( -\half i\theta_k^{(n)} \right) - \exp\left( \half i\theta_k^{(n)} \right) }\\
                      & = \prod_{k=1}^n \exp( \frac{i \pi z }{n} ) 
                                        \frac{ \sin\left( \frac{ \pi z }{n} - \half \theta_k^{(n)} \right) }
                                             { \sin\left( -\half \theta_k^{(n)} \right) }\\
                      & = \exp( i \pi z ) \prod_{k=1}^n \frac{ \sin\left( \half \theta_k^{(n)} - \frac{ \pi z }{n} \right) }
                                                             { \sin\left( \half \theta_k^{(n)} \right) }\\
\end{align*}
Now, the standard product formula for the sine function can be written as follows:
$$ \forall \alpha \in \C, \sin\left( \alpha \right) = \alpha \underset{A \rightarrow \infty}{\lim} 
\prod_{0  < |j| \leq A} \left( 1 - \frac{\alpha}{\pi j} \right).$$
We then have: 
\begin{align*}
\xi_n\left( z \right) & = \exp( i \pi z ) \prod_{k=1}^n \left( 
\							\frac{\half \theta_k^{(n)} - \frac{ \pi z }{n}}{\half \theta_k^{(n)}}
						\underset{A \rightarrow \infty}{\lim} \prod_{0 < |j| \leq A}
						\frac{ 1 - \frac{\half \theta_k^{(n)} - \frac{ \pi z }{n}}{\pi j} }
                                                                               { 1 - \frac{\half \theta_k^{(n)} }{\pi j} }
							\right)\\
                      & = \exp( i \pi z ) \prod_{k=1}^n \left( 
							\left( 1 - \frac{z }{y_k^{(n)}} \right)
						\underset{A \rightarrow \infty}{\lim} 	\prod_{0 < |j| \leq A} \left( 1 - \frac{z}{nj + y_k^{(n)}} \right)
							\right)\\
                      & = \exp( i \pi z ) \prod_{k=1}^n \underset{A \rightarrow \infty}{\lim} 	\prod_{0 \leq |j| \leq A} 
                     \left( 1 - \frac{z}{nj + y_k^{(n)}} \right)
\end{align*}
Using the periodicity of the eigenangles, we have:
$$ y_{k + jn}^{(n)} = jn + y_{k}^{(n)},$$
and then 
$$\xi_n\left( z \right) =  \exp( i \pi z ) \underset{A \rightarrow \infty}{\lim} \prod_{1-nA \leq k \leq n + nA}
\left( 1 - \frac{z}{y_k^{(n)}}\right).$$
Now, for $B \geq 2n$, $A \geq 2$ integers such that $An \leq B \leq An + n -1$,
the product of $1 - \frac{z}{y_k^{(n)}}$ from $1-nA$ to $n+nA$ and the product from $-B$ to $B$ differ by at most 
$2n$ factors, which are all $1 + O(|z|/y_{nA}) + O(|z|/|y_{1 -nA}|) = 1 + O(|z|/nA)$.  
The quotient between these two products is then well-defined and $\exp [O(|z|/A)] = \exp[O(n|z|/B)]$ for
$B$ large enough, which implies 
that it tends to one when $B$ goes to infinity. Hence, 
$$\xi_n\left( z \right) =  \exp( i \pi z ) \underset{B \rightarrow \infty}{\lim} \prod_{-B \leq k \leq B}
\left( 1 - \frac{z}{y_k^{(n)}}\right).$$
\end{proof}

We are now ready to prove Theorem \ref{thm:main}. 
\begin{proof}[Proof of theorem \ref{thm:main}]
Thanks to the estimate from Proposition \ref{proposition:key_estimate}:
$$ y_k^{(n)} = k + O\left( \log(2+|k|) \right)$$
We have that, for $k\geq 1$ and $z$ in a compact $K$:
\begin{align*}
\left( 1 - \frac{z}{y_k^{(n)}}\right)\left( 1 - \frac{z}{y_{-k}^{(n)}}\right) & = 1 - z \frac{O(\log(2+|k|))}{k^2} + O\left( \frac{z^2}{k^2} \right)\\
& = 1 + \frac{ O_{K}\left( \log(2+|k|) \right) }{k^2}
\end{align*}
Hence:
$$ \xi_n\left( z \right) =  e^{i \pi z} \prod_{k \in \Z}\left( 1 - \frac{z}{y_k^{(n)}}\right)$$
is a sequence of entire functions uniformly bounded on compact sets. Therefore, by Montel's theorem uniform 
convergence on compact sets is implied by pointwise convergence. Let us then focus on proving pointwise convergence.

Fix $A \geq 2$. Let us prove that:
\begin{equation}
 \prod_{|k| \leq A}\left( 1 - \frac{z}{y_k^{(n)}}\right) - \prod_{k \in \Z}\left( 1 - \frac{z}{y_k^{(n)}}\right) = 
 O_{K}\left( \frac{\log A}{A} \right), \label{prodkAn}
 \end{equation}
 \begin{equation}
 \prod_{|k| \leq A}\left( 1 - \frac{z}{y_k      }\right) - \prod_{k \in \Z}\left( 1 - \frac{z}{y_k  
 }\right) = O_{K}\left( \frac{\log A}{A} \right). \label{prodkA}
 \end{equation}
Here, the infinite products are, as before, the limits of the products from $-B$ to $B$ for $B$ going to infinity. 
Note that the existence of the infinite product involving $y_k$ is an immediate consequence of the absolute 
convergence of the product
$$\left( 1 - \frac{z}{y_0} \right) \prod_{k \geq 1} \left[ \left( 1 - \frac{z}{y_k} \right)
\left( 1 - \frac{z}{y_{-k}} \right) \right],$$
 stated in Theorem \ref{thm:main}, and following from the estimate:  
 $$\left( 1 - \frac{z}{y_k} \right)
\left( 1 - \frac{z}{y_{-k}} \right) =  1 - z \frac{O(\log(2+|k|))}{k^2} + O\left( \frac{z^2}{k^2} \right)
= 1 + \frac{ O_{K}\left( \log(2+|k|) \right) }{k^2}.$$

We now prove \eqref{prodkAn}: a proof of \eqref{prodkA} is simply obtained by removing the indices $n$. We have:
$$ \prod_{|k| \geq A}\left( 1 - \frac{z}{y_k^{(n)}}\right) = 1 +  O_{K}\left( \sum_{k \geq A} \frac{ \log(2+|k|) }{k^2} \right) = 1 +  O_{K}\left( \frac{\log A}{A} \right)$$
and
$$ \prod_{|k| \leq A}\left( 1 - \frac{z}{y_k^{(n)}}\right) = O_{K}\left( 1 \right)$$

Therefore:
\begin{align*}
  & \prod_{|k| \leq A}\left( 1 - \frac{z}{y_k^{(n)}}\right) - \prod_{k \in \Z}\left( 1 - \frac{z}{y_k^{(n)}}\right)\\
= & \prod_{|k| \leq A}\left( 1 - \frac{z}{y_k^{(n)}}\right)\left( 1 - \prod_{|k| > A}\left( 1 - \frac{z}{y_k^{(n)}}\right) \right)\\
= & \prod_{|k| \leq A}\left( 1 - \frac{z}{y_k^{(n)}}\right)\left( 1 - \left(1 +  O_{K}\left( \frac{\log A}{A} \right) \right) \right)\\
= & O_{K}\left( \frac{\log A}{A} \right)
\end{align*}
Because errors are uniform in $n$, this is saying:
$$ \sup_{n} \left| \prod_{|k| \leq A}\left( 1 - \frac{z}{y_k^{(n)}}\right) - \prod_{k \in \Z}\left( 1 - \frac{z}{y_k^{(n)}}\right)\right| \underset{A \rightarrow \infty}{\longrightarrow} 0$$

Now:
\begin{align*}
     & \left| \prod_{k \in \Z}\left( 1 - \frac{z}{y_k^{(n)}}\right) - \prod_{k \in \Z}\left( 1 - \frac{z}{y_k}\right)\right|\\
\leq &    \left| \prod_{|k| \leq A}\left( 1 - \frac{z}{y_k^{(n)}}\right) - \prod_{|k| \leq A}\left( 1 - \frac{z}{y_k}\right)\right|\\
     &  + \left| \prod_{k \in \Z}\left( 1 - \frac{z}{y_k^{(n)}}\right)   - \prod_{|k| \leq A}\left( 1 - \frac{z}{y_k^{(n)}}\right)\right| \\
     &  + \left| \prod_{k \in \Z}\left( 1 - \frac{z}{y_k}\right)         - \prod_{|k| \leq A}\left( 1 - \frac{z}{y_k}\right)\right| \\
\leq &  \left| \prod_{|k| \leq A}\left( 1 - \frac{z}{y_k^{(n)}}\right) - \prod_{|k| \leq A}\left( 1 - \frac{z}{y_k}\right)\right| + O_{K}\left( \frac{\log A}{A} \right)
\end{align*}

Hence, as $y_k^{(n)} \rightarrow y_k$ pointwise:
$$ \limsup_{n \rightarrow \infty}\left| \prod_{k \in \Z}\left( 1 - \frac{z}{y_k^{(n)}}\right) - \prod_{k \in \Z}\left( 1 - \frac{z}{y_k}\right)\right| = O_{K}\left( \frac{\log A}{A} \right)$$

Taking $A \rightarrow \infty$ completes the proof.
\end{proof}

\section{Properties of the limiting function of \texorpdfstring{$\xi_\infty$}{xi infty} }
\label{section:properties}
In this section, we establish some properties of $\xi_\infty$ and then provide some questions which would help acquiring a deeper understanding of $\xi_\infty$.

We first start with a simple statement on the order of $\xi_\infty$ as an entire function:

\begin{proposition}
Almost surely, $\xi_\infty$ is of order $1$. More precisely, the exists a.s. a random 
$C > 0$, such that for all $z \in \mathbb{C}$. 
$$ |\xi_{\infty}(z)| \leq e^{C|z| \log(2 + |z|)}.$$
On the other hand, there exists a.s. a random $c > 0$ such that for all $x \in \mathbb{R}$, 
$$|\xi_{\infty}(ix)| \geq c e^{c|x|}.$$
\end{proposition}
\begin{proof}
We have:
$$ \left( 1 - \frac{z}{y_k}\right)\left( 1 - \frac{z}{y_{-k}}\right) = 1 - z \frac{O(\log(2+|k|))}{k^2} + O\left( \frac{z^2}{k^2} \right) $$
with errors being uniform in $z$ and $k \geq 1$. We distinguish between three regimes for $k \in \Z$ different
from zero:
$|k| \geq e^{|z|} $, $|z| \leq  |k| < e^{|z|}$, $1 \leq |k| < |z|$. 
In the first regime, 
$$\left( 1 - \frac{z}{y_k}\right)\left( 1 - \frac{z}{y_{-k}}\right) =  1 + 
O \left( \frac{|z|(\log(2+|k|))}{k^2} \right),$$
which implies 
\begin{align*} \left| \prod_{k \geq e^{|z|}} \left( 1 - \frac{z}{y_k}\right) \, \left( 1 - \frac{z}{y_{-k}}\right)
\right|
& \leq \exp \left( O \left( |z| \sum_{k \geq e^{|z|}} \frac{\log(2 + k)}{k^2} \right) \right)
\\ & = \exp \left( O \left( |z| \sum_{k \geq e^{|z|}} k^{-3/2} \right) \right)
\\ & =  \exp \left( O \left( |z|e^{-|z|/2}  \right) \right) = O(1). 
\end{align*}
In the second regime, $$\log (2 + |k|) \leq \log (e^{|z|} + 2) \leq \log(3 e^{|z|}) \leq |z| + 2,$$
and then 
$$\left( 1 - \frac{z}{y_k}\right)\left( 1 - \frac{z}{y_{-k}}\right) = 1 + O \left( \frac{|z|(|z|+2)}{k^2}
\right),$$
which implies 
$$ \left| \prod_{|z| \leq k <  e^{|z|}} \left( 1 - \frac{z}{y_k}\right)\left( 1 - \frac{z}{y_{-k}}\right) 
\right|
\leq \exp \left( O \left( |z|(|z|+2)  \sum_{k \geq |z| \vee 1} \frac{1}{k^2} \right) \right)
 = \exp  O(|z|).$$
 Finally, in the third regime, we have, since $|y_k/k|$ is a.s. bounded from below, 
 $$1 - \frac{z}{y_k} = 1 + O (|z/k|),$$
 which in turn implies 
 $$  \left| \prod_{1 \leq k < |z|} \left( 1 - \frac{z}{y_k}\right)\left( 1 - \frac{z}{y_{-k}}\right) 
\right| \leq \exp \left( O \left(|z|  \sum_{1 \leq k < |z|} (1/k) \right) \right) 
= \exp   O \left(|z| \log( 2 + |z|) \right) .$$
 Since 
 $$\left|1 - \frac{z}{y_0} \right| \leq \exp (|z|/ y_0) = \exp O(|z|),$$
 we deduce by combining the three regimes, the  following upper bound: 
 $$|\xi_{\infty} (z)| \leq \exp  O \left(|z| \log( 2 + |z|) \right).$$
  In order to prove the lower bound, we first use the equality: 
  $$ |\xi_{\infty}(ix)|^2 = \prod_{k \in \mathbb{Z}} \left( 1 + \frac{x^2}{y^2_k} \right).$$
 Since $|y_k| = O(|k|)$ for $k \neq 0$, we deduce that there exists a random $c > 0$ such that 
 $$|\xi_{\infty}(ix)|^2  \geq  \prod_{k \neq 0} \left( 1 + \frac{x^2}{c k^2} \right),$$
 and then 
 $$|\xi_{\infty}(ix)| \geq \prod_{k \geq 1} \left( 1 + \frac{x^2}{c k^2} \right) = \frac{\sinh(\pi x/\sqrt{c})}
 {\pi x/ \sqrt{c}},$$
 which shows the lower bound given in the proposition. 

\end{proof}

One might also naturally be interested in the finite dimensional distribution of the random analytic function $\xi_\infty$:
\begin{question}
What can be said about the joint law of $n$ points:
$$ \left( \xi_\infty(z_1), \dots, \xi_\infty(z_n) \right)$$
or the law of the coefficients in the Taylor expansion near zero?
\end{question}

A possible approach to this question would  be to compute the law of the power polynomials:
$$ P_\alpha := \sum_{k \in \Z} \frac{1}{y_k^\alpha}$$
which appear in the series expansion of $\log \xi_\infty$ near zero.

\begin{rmk}
All series for $\alpha \geq 2$ converge absolutely (almost surely). For $\alpha=1$, one has of course to take the symmetric sum.
\end{rmk}

Using standard tools related to the CUE, we can express characteristic functions of power polynomials thanks to Toeplitz determinants. Let us start by a simple lemma:

\begin{lemma}
For every integer $\alpha \in \N$, define the rational function $R_\alpha$ as:
$$ R_\alpha(X) = \left( X \frac{d}{dX} \right)^{\alpha} \frac{X+1}{X-1}$$
Then, we have:
 $$ \sum_{k \in \Z} \frac{1}{\left( x + 2 \pi k \right)^{\alpha+1}} = 
 \frac{i^{\alpha+1}}{2} \frac{(-1)^{\alpha}}{\alpha!} R_\alpha\left( e^{ix} \right)$$
\end{lemma}
\begin{proof}
 Using the series expansion for the cotangent function, which converges uniformly under pairwise summation:
 \begin{align*}
 \sum_{k \in \Z} \frac{1}{\left( x + 2\pi k \right)^{\alpha+1}} & = \frac{(-1)^\alpha}{\alpha!} \frac{d^{\alpha}}{dx^{\alpha}} \sum_{k \in \Z} \frac{1}{x + 2\pi k }\\
 & = \frac{1}{2} \, \frac{(-1)^{\alpha}}{\alpha!} \frac{d^{\alpha}}{dx^{\alpha}} \cot(\frac{x}{2})
 \end{align*}
 Now since
 $$ \frac{1}{2} \cot(\frac{x}{2}) = \frac{i}{2} \frac{e^{ix}+1}{e^{ix}-1}$$
 we have:
 $$  \sum_{k \in \Z} \frac{1}{\left( x + 2 \pi k \right)^\alpha} = \frac{i}{2} \frac{(-1)^{\alpha}}{\alpha!} \frac{d^{\alpha}}{dx^{\alpha}} \frac{e^{ix}+1}{e^{ix}-1}$$
 Now
 $$ \frac{d}{dx} = i e^{ix} \frac{d}{de^{ix}}$$
 completes the proof.
\end{proof}

We now note:
\begin{align*}
P_{\alpha+1} & = \lim_{n \rightarrow \infty} \sum_{k \in \Z} \frac{1}{\left( y_k^{(n)} \right)^{\alpha+1}}\\
& = \lim_{n \rightarrow \infty} \sum_{k=1}^n \sum_{l \in \Z} \frac{1}{\left( y_k^{(n)} + nl \right)^{\alpha+1}}\\
& = \lim_{n \rightarrow \infty} \sum_{k=1}^n \sum_{l \in \Z} \frac{1}{\left( \frac{n \theta_k^{(n)}}{2\pi} + nl \right)^{\alpha+1}}\\
& = \lim_{n \rightarrow \infty} \left( \frac{2\pi}{n} \right)^{\alpha+1} \sum_{k=1}^n \sum_{l \in \Z} \frac{1}{\left( \theta_k^{(n)} + 2 \pi l \right)^{\alpha+1}}\\
& = \lim_{n \rightarrow \infty} \left( \frac{2\pi}{n} \right)^{\alpha+1} \frac{i^{\alpha+1}}{2} \frac{(-1)^{\alpha}}{\alpha!} \sum_{k=1}^n R_\alpha\left( e^{i \theta_k^{(n)}} \right)\\
& = - \frac{1}{2 \alpha!} \lim_{n \rightarrow \infty} \left( -\frac{2i\pi}{n} \right)^{\alpha+1}  \sum_{k=1}^n R_\alpha\left( e^{i \theta_k^{(n)}} \right)\\
\end{align*}

We recognize a limit of linear statistics on the eigenvalues of the CUE. Now using the theory of Toeplitz determinants, we have
$$ \E\left( \exp( - 2 i\lambda P_{\alpha+1}) \right) = \lim_{n \rightarrow \infty} D_n( \phi_n )$$
where:
$$ D_n( \phi_n ) = \E\left( e^{ \left( -\frac{2i\pi}{n} \right)^\alpha \frac{\lambda}{\alpha!} Tr R_\alpha(U_n) } \right)$$
is the Toeplitz determinant with symbol:
$$ \phi_n(x) = e^{ \left( -\frac{2i\pi}{n} \right)^\alpha \frac{\lambda}{\alpha!} R_\alpha(e^{ix}) }$$

Notice that the symbol changes with $n$, making the evaluation more delicate. Moreover, the symbol has poles as singularities.

Now we would like to conclude with a question which has to do with our initial motivation:
\begin{question}
Is there a random version of the Keating Snaith moments conjecture?
$$ \frac{1}{T}\int_0^T \left| \xi_\infty( \frac{t}{\log T} )\right|^{2k} dt \stackrel{T \rightarrow \infty}{\sim} \left( \log T \right)^{k^2} g_k$$
\end{question}

\bibliographystyle{halpha}
\bibliography{Bib_CueCharPolyCV2}

\end{document}